\documentclass[times,11pt]{elsarticle}

\usepackage{amssymb,amsmath,amsthm, color}

\newtheorem{theorem}{Theorem}[section]
\newtheorem{lemma}[theorem]{Lemma}
\newtheorem{define}[theorem]{Definition}
\newtheorem{coro}[theorem]{Corollary}
\newtheorem{propo}[theorem]{Proposition}
\newtheorem{remark}[theorem]{Remark}

\newcommand{\ff}[1]{{\mathbb F}_{\! #1}}
\newcommand{\fff}[1]{{\overline{\mathbb F}}_{\! #1}}

\newcommand{\N}{\mathbb N}
\newcommand{\Z}{\mathbb Z}

\newcommand{\GL}{\mathrm{GL}}
\newcommand{\PGL}{\mathrm{PGL}}

\newcommand{\ord}{\mathrm{ord}}
\newcommand{\an}[1]{\langle #1\rangle}

\begin{document}

\begin{frontmatter}



\title{On the multiplicative order of the roots of $bX^{q^r+1}-aX^{q^r}+dX-c$}

\author[br]{F.E. Brochero Mart\'{i}nez}
\ead{fbrocher@mat.ufmg.br}
\author[gr]{Theodoulos Garefalakis}
\ead{tgaref@uoc.gr}
\author[br]{Lucas Reis}
\ead{lucasreismat@gmail.com}
\author[gr]{Eleni Tzanaki}
\ead{etzanaki@uoc.gr}

\address[br]{Departamento de Matem\'{a}tica, Universidade Federal de Minas Gerais, Belo Horizonte, MG, 30123-970, Brazil}
\address[gr]{Department of Mathematics and Applied Mathematics, University of Crete, 70013 Heraklion, Greece}


\begin{abstract}
In this paper, we find a lower bound for the order of the group $\langle \theta+\alpha\rangle \subset \fff{q}^*$, where $\alpha\in \ff{q}$, $\theta$ is a generic root of the polynomial 
$F_{A,r}(X)=bX^{q^r+1}-aX^{q^r}+dX-c\in \ff{q}[X]$ and $ad-bc\ne0$.
\end{abstract}

\end{frontmatter}





\section{Introduction}
Let $\ff{q}$ be the field with $q$ elements, where $q$ is a power of a prime 
$p$. Given a positive integer $n$, it is natural to ask how to find elements of 
very high order in the multiplicative group 
$\left(\frac{\ff{q}[X]}{f(x)}\right)^*$, where $f(x)$ is an irreducible 
polynomial of degree $n$.    
Elements of this type are used in the AKS algorithm (see \cite{aks}), for 
determining primality in polynomial time. 
This question is closely related to the problem of efficiently  
constructing  a primitive element of a given finite field, which has practical 
applications in Coding Theory and Cryptography.
This last  problem has been considered by many authors: In \cite{Gao}, Gao 
gives an algorithm for explicitly constructing  elements for a general 
extension $\ff{q^n}$ of the field $\ff{q}$, with order bounded below by a 
function of the form $\exp \left(c(p)\frac {\log^2\log q}{\log\log\log 
q}\right)$, where $c(p)$ 
depends only  on the characteristic of the field.
 In \cite{Cheng}, Cheng shows how to find, given $q$ and $N$,     an integer $n$ 
 in the interval $[N,2qN]$, and a $\theta$ in the field $\ff{q^n}$ with order 
 larger than $5.8^{n\log q/\log n}$. 
 In \cite{Popovych1} and \cite{Popovych2}, Popovych considers the case where  
 $f(X)=\Phi_r(X)$, the $r$-th cyclotomic polynomial,  and  $f(X)=X^n-a$ are 
 irreducible polynomials in $\ff{q}[X]$  and finds a lower bound of the order 
 of  $\langle \theta+c\rangle$, where $\theta$ is a root of $f(X)=0$.  
Finally in \cite{MartinezReis2016}, the authors consider the same problem  with the polynomial $f(X)=X^p-X+c\in \ff{q}[X]$.

On the other hand, in \cite{ST12}, Stichtenoth and Topuzo\v{g}lu show that, given a matrix $[A]=\left[\begin{pmatrix} a&b\\ c&d\end{pmatrix}\right]\in \PGL_2(\ff{q})$, 
every  irreducible factor $f$ of $F_{A,r}(X)=bX^{q^r+1}-aX^{q^r}+dX-c$ in $\ff{q}[X]$ is invariant by an appropriate natural action of $[A]$ and reciprocally, every irreducible polynomial $f$,
 invariant by the action of $[A]$, is a factor of $F_{A,r}(X)$ for some $r\ge 0$.
 This relation is used in \cite{ST12} to estimate, asymptotically,  the number of irreducible monic polynomial of given degree and invariant by $[A]$ and they conclude that, in general, the irreducible factors of $F_{A, r}(X)$ has degree $Dr$, where $D$ is the order of $[A]$ in $ \PGL_2(\ff{q})$.

In this paper we study the problem of finding elements of high order arising from fields $\left(\frac{\ff{q}[X]}{f(X)}\right)^*$, where $f(X)$ is an irreducible factor of $F_{A, r}(X)$ and we obtain the following:
\begin{theorem}\label{main} Let $\alpha\in \ff{q}$, $A \in \GL_2(\ff{q})$, $[A]\ne [I]$ and $\theta$ be a generic  root of $F_{A,r}$, i.e. $\theta\in \fff{q}$ satisfies $\dim_{\ff{q}}\ff{q}[\theta]=Dr$ where  $D=\ord([A])$ and $r>2$.  The multiplicative order of $\theta+\alpha$ is bounded below by
\begin{equation}\frac1{\sqrt{2}\pi D} \sqrt{\frac {r-2}{r+2}}\cdot\left(\frac { (r+2)^{r+2}}{(r-2)^{r-2}}\right)^{\frac D4}  \exp\left(-\frac 5{24D}\cdot\frac {r^2+4}{r^2-4}\right),
\end{equation}
in the case that $(1,0)$ and $(0,1)A^j$ are linearly independent for all $j$ and
\begin{equation}\frac{\sqrt{2}}{\pi D} \sqrt{\frac {r-1}{r+1}}\cdot\left(\frac { 4(r+1)^{r+1}}{(r-1)^{r-1}}\right)^{\frac D4}  \exp\left(-\frac 1{24D}\cdot\frac {5r^2+3}{r^2-1}\right),
\end{equation}
otherwise.
\end{theorem} 

\begin{remark}  For every $\epsilon>0$ and $r>R_\epsilon$, the lower bound (1) is greater than 
$$\frac 1{\sqrt 2\pi D} ((e-\epsilon)(r+2))^D$$
and  the lower bound (2) is greater than
$$\frac{\sqrt 2}{\pi D} (2 (e-\epsilon)(r+1))^{D/2}.$$
\end{remark}

We note that, $\theta$ is a root of $F_{A,r}$ if and only if $\theta+\alpha$ is root of $F_{B,r}$, where
$$B=\begin{pmatrix}a+b\alpha&b\\ c+d\alpha-a\alpha-b\alpha^2&d-b\alpha \end{pmatrix}\in \GL_2(\ff{q}),$$ 
and the matrices $A$ and $B$ have the same eigenvalues, hence their multiplicative order are the same. Since our bounds essencially depend of the order of $A$ and $r$, in the following, unless otherwise stated, we assume that $\alpha=0$.   

In  addition, in the case when $A$ is a triangular matrix,  this lower bound can also be  improved.

\section{Preliminaries}
Throughout  this paper,  $\ff{q}$ is the finite field with $q$ elements, where $q$ is a power
of a prime $p$; given a matrix $A=\begin{pmatrix} a&b\\ c&d\end{pmatrix}\in \GL_2(\ff{q})$,  $[A]$ denotes its class in 
$\PGL_2(\ff{q})$ and $D=\ord([A])$. Observe that, in the case   $\det(A)=1$ and $A$ is diagonalizable, the eigenvalues of $A$ are $\gamma$ and $\gamma^{-1}$ and we have that  $
D=\ord([A]) = \frac{\ord \gamma}{(\ord \gamma, 2)}$ and then $A^D=(-1)^{D+1} I$.
 In addition,  for each  non-negative integer $r$,   $F_{A,r}(X)$ denotes the polynomial 
$ bX^{q^r+1}-aX^{q^r}+dX-c$. 

There is an action of the general linear group $\GL_2(\ff{q})$
on the set of irreducible polynomials of degree at least 2, which was 
studied in \cite{Gar11,ST12}. In this work, we adopt the notation of \cite{ST12}. 
\begin{define}
Let $\displaystyle
A=\left( \begin{array}{cc}a&b\\ c&d\end{array}\right)\in \GL_2(\ff{q})$. 
For an irreducible polynomial $f(X)\in\ff{q}[X]$ of degree $n\geq 2$ and  $\theta\in \fff{q}\setminus\ff{q}$,  define
\begin{enumerate}
\item $\displaystyle 
(A \circ f)(X) := (bX+d)^n\cdot f\left(\frac{aX+c}{bX+d}\right).$
\item $[A] \circ f(X):=$ the unique monic polynomial $g(X)$ such that $  (A \circ f)(X)=\lambda g(X)$ for some $\lambda\in \fff{q}$.
\item 
$\displaystyle  [A]\circ\theta= A\circ \theta := \frac{d \theta-c}{-b\theta +a}
$.
\end{enumerate}
\end{define}
It turns out that the above rules define actions of $\GL_2(\ff{q})$ on 
the set of irreducible polynomials of degree at least 2 in $\ff{q}[X]$
and on $\fff{q}\setminus \ff{q}$ respectively and these actions are closely
related: from Lemma 2.7 in \cite{ST12}, it  follows that $\theta$ is a root of $f$ if and only if $A\circ \theta$ is a root of $A \circ f$.

One of the goals of \cite{ST12} is the characterization and counting the monic irreducible
polynomials that are fixed by the action of a given matrix. The following theorems provide such a characterization.
\begin{theorem}[\cite{ST12}, Theorems 4.2 and  4.5 ]
Let $f(X)\in \ff{q}[X]$ be a monic irreducible polynomial of degree $n\ge 2$. The following are equivalent:
\begin{enumerate}
	\item $[A]\circ f = f$
	\item $f\ |\ F_{A,r}$ for some   non-negative integer $r<n$.
\end{enumerate}
In addition, every irrreducible factor of $F_{A,r}$ has degree $\le 2$ or $Dk$, where $k|r$ and $\gcd(\frac rk,D)=1$.  
\end{theorem}
Expecifically, denoting
\[
 N_{A,r}(n)=\left|\left\{f\in\ff{q}[X]\ :\ f \mbox{ monic, irreducible }, \deg(f)=n, f|F_{A,r}\right\}\right|,
\]
it follows that
\begin{theorem}[\cite{ST12}, Theorems 5.2]
Let $A\in \GL_2(\ff{q})$ and $\ord([A])=D\geq 2$. Then 
\begin{enumerate}
 \item $N_{A,r}(n) = 0$, if $D \nmid n$, $n\geq 2$.
 \item $N_{A,r}(Dr)\sim \frac{q^r}{Dr}$, as $r\rightarrow \infty$,
\end{enumerate}
that is, all non-linear irreducible factors of $F_{A,r}$ have degree divisible by $D$
and almost all have degree $Dr$, as $r$ tends to infinity.
\end{theorem}

In order to bound the order of a generic root $\theta$ of the polynomial $F_{A,r}(X)$, i.e. $\theta$ is a root of  $F_{A,r}(X)$ such that $\dim_{\ff{q}}\ff{q}[\theta]=Dr$, it is enough to find a set $J\subset \N$ such that $\theta^i\ne \theta^j$ for every  $i\ne j$ elements of $J$ and thus $\ord (\theta)\ge |J|$. In order to find such set, observe that $\theta$ satisfies  the relation
 $\theta^{q^r} =  A\circ \theta,
$
and inductively we obtain that
\begin{equation}\label{eq:A-to-power}
 \theta^{q^{jr}} = A^j \circ \theta,\ \ \mbox{ for } j\in \Z_{\geq 0}. 
\end{equation}
The main idea lies on the construction of an appropriate set $J$ having elements of the form
$ u_0+u_1q^r+\cdots +u_{D-1} q^{r(D-1)}$, with some restriction on $u_j\in \Z$, and use the relation (\ref{eq:A-to-power}) to show that the elements in $\{ \theta^j, j\in J\}$ are all different.

In order to prove Theorem \ref{main}, we need the following  technical lemmas:

\begin{lemma}\label{lemma:LI-1}
Let $A=\left( \begin{array}{cc}a&b\\ c&d\end{array}\right)\in \GL_2(\fff{q})$, with $\det(A)=1$ and $bc\neq 0$ .
Let  us 
denote $(a_n,b_n)$ and $(c_n,d_n)$ the first and second
row, respectively, of $A^n$, $n\in\N$. Then for any $0\leq k<n<D$, the vectors $(a_n,b_n), (a_k,b_k)$
are linearly independent over $\fff{q}$. The same holds for the vectors $(c_n,d_n),(c_k,d_k)$. 
\end{lemma}
\begin{proof}
Let us suppose that $A$ is a diagonalizable matrix and denote by  $\alpha,\alpha^{-1}$ the two eigenvalues of $A$. 
Since $A$ is a diagonalizable matrix,  we can write
\[
A= M \left( \begin{array}{cc}\alpha&0\\ 0&\alpha^{-1}\end{array}\right) M^{-1},
\quad\text{
where}
\quad
M = \left( \begin{array}{cc}t&u\\ v&w\end{array}\right)
\]
is an invertible matrix . The assumption $bc\neq 0$ implies $tuvw\neq 0$. 

By direct calculation, we have that
\[
A^n=\left( \begin{array}{cc}\delta(tw\alpha^n-uv\alpha^{-n}) &\delta ut(\alpha^{-n}-\alpha^n)\\
                \delta vw(\alpha^n-\alpha^{-n})& \delta(wt\alpha^{-n}-uv\alpha^{n})\end{array}\right),\ \ n\in\N.
\]
where $\delta:=(tw-uv)^{-1}=(\det(M))^{-1}$.
Let us suppose that $(a_n,b_n)=\gamma(a_k,b_k)$ for some $0\leq k<n<D$ and some $\gamma\in \fff{q}$, then
\begin{eqnarray*}
  tw\alpha^n-uv\alpha^{-n} &=& \gamma(tw\alpha^k-uv\alpha^{-k}) \\
  ut(\alpha^{-n}-\alpha^n) &=& \gamma ut(\alpha^{-k}-\alpha^k),
\end{eqnarray*}
which implies
\begin{eqnarray*}
tw (\alpha^n - \gamma \alpha^k) &=& uv (\alpha^{-n}-\gamma \alpha^{-k}) \\
\alpha^n - \gamma \alpha^k &=& \alpha^{-n}-\gamma \alpha^{-k}.
\end{eqnarray*}
If $\alpha^n\neq \gamma \alpha^k$, we obtain $tw=uv$, a contradiction since $M$ is invertible. Therefore 
$\alpha^n=\gamma\alpha^k$ and $\alpha^{-n}=\gamma \alpha^{-k}$, hence $\alpha^{2(n-k)}=1$, i.e., $\ord(\alpha)$ divides $2(n-k)$. If $\ord(\alpha)$ is even, then $2D=\ord(\alpha)$ and $0<2(n-k)<2D$.
If $\ord(\alpha)$ is odd, then $\ord(\alpha)$ divides $(n-k)$, $D=\ord(\alpha)$ and $0<n-k<D$. Both cases lead us to a contradiction. The proof of the linear independence of $(c_n,d_n)$ and $(c_k,d_k)$ follows similarly.

When $A$ is non diagonalizable matrix, then 
$$A=M^{-1}\begin{pmatrix} 1&0\\ 1&1\end{pmatrix}M,\quad\text{ where }\quad M=\left( \begin{array}{cc}t&u\\ v&w\end{array}\right)$$ and 
\[
A^n=\left( \begin{array}{cc}1-n\delta tu &-n\delta u^2\\
                n\delta t^2& 1+n\delta tu\end{array}\right),\ \ n\in\N.
\]
By the same process of the diagonalizable case, we conclude the proof.
\end{proof}

\begin{lemma}\label{lemma:LI-3} Let  $A=\left( \begin{array}{cc}a&0\\ c&d\end{array}\right)\in \GL_2(\fff{q})$ with   $c\ne 0$ and $(c_n, d_n)$ as in the previous lemma. Then for any $0\leq k<n<D$, the vectors $(c_n,d_n),(c_k,d_k)$ are linearly independent over $\fff{q}$.
\end{lemma}
 
\begin{proof} By a direct calculation, we have that
$$A^n=
\begin{pmatrix} a^n &0\\ c\frac {a^n-d^n}{a-d}& d^n\end{pmatrix}\quad
\text{if $a\ne d$}
$$ 
and
$$
A^n= \begin{pmatrix} a^n &0\\ nca^{n-1}& a^n\end{pmatrix}\quad
\text{if $a=d$.}
$$
Let us suppose that $(c_n,d_n)=\gamma(c_k,d_k)$ for some $0\leq k<n<D$ and some $\gamma\in \fff{q}$, in the case $a\ne d$, it follows that $\gamma=d^{n-k}$ and
$$ c\frac {a^n-d^n}{a-d}= cd^{n-k}\frac {a^k-d^k}{a-d}.$$
Since $c\ne 0$, we obtain that $a^{n-k}=d^{n-k}$ and therefore $A^{n-k}=a^{n-k} I$, which is impossible since $0<n-k<D$.  The second case is similar. 
\end{proof}
\begin{remark} When $A=\left( \begin{array}{cc}a&0\\ c&d\end{array}\right)\in \GL_2(\fff{q})$ is a triangular matrix, $[A]\ne [I]$, then 
$$\ord([A])=\begin{cases} \ord(\frac ad)&\text{if $a\ne d$}\\
p&\text{if $a= d$ and $c\ne 0$.}
\end{cases}$$
\end{remark}

\begin{lemma}\label{lemma:LI-2}
Let $A=\left( \begin{array}{cc}a&b\\ c&d\end{array}\right)\in \GL_2(\fff{q})$ 
and denote by $(a_n,b_n)$ and $(c_n,d_n)$ the first and second
row, respectively, of $A^n$, $n\in\N$. Assume that $(c_n,d_n) = \gamma (a_k,b_k)$ for some $0\leq k,n<D$
and $\gamma\in\fff{q}$. Then, denoting $g=n-k$, we have
\[
  (c_{i},d_{i}) = \epsilon_i \gamma (a_{i-g},b_{i-g}), \ \ 0\leq i\leq D-1,
\]
where $\epsilon_i\in \{-1,1\}$ and the indexes  are computed modulo $D$.
\end{lemma}
\begin{proof}
By definition, $(a_k, b_k)=(1,0)A^k$ and $(c_n,d_n)=(0,1) A^n$,
hence $(0,1) A^{g}=\gamma (1,0)$, where $g=n-k$. Therefore $(0,1) A^{g+i}=\gamma (1,0)A^i$, that is, 
\begin{equation}\label{eq:aux-1}
(c_{g+i},d_{g+i}) = \gamma (a_i,b_i), \ \ \forall i\geq 0.
\end{equation}
Assume $k<n$. From this it follows that
\begin{eqnarray*}
  (c_{g+i},d_{g+i}) &=& \gamma (a_i,b_i), \ \ i=0\ldots, D-g-1, \\
  (c_{D+i},d_{D+i}) &=& \gamma (a_{D-g+i}, b_{D-g+i}), \ \ i=0,\ldots, g-1,
\end{eqnarray*}
where the second identity follows by changing $D-g+i$ for $i$ in Eq.~\eqref{eq:aux-1}.
Now, since $A^D=(-1)^{D+1}I$ we have that $(c_{D+i},d_{D+i})=(0,1)A^{D+i} = (-1)^{D+1} (c_i,d_i)$, so we have
\begin{eqnarray*}
  (c_{i},d_{i}) &=& \gamma (-1)^{D-1} (a_{D-g+i}, b_{D-g+i}), \ \ i=0,\ldots,g-1, \\
  (c_{i},d_{i}) &=& \gamma (a_{i-g},b_{i-g}), \ \ i=g,\ldots, D-1. \\
\end{eqnarray*}
If $k>n$ the computation is entirely similar and the case $k=n$ is not possible since $(a_k, b_k)$ and $(c_k, d_k)$ are linearly independent.
\end{proof}

\begin{remark}
If $\rho$ is the smallest prime factor of $D$ and $g$ is defined as in Lemma \ref{lemma:LI-2}, it is clear that $$(g, D)\le D/\rho$$ and this bound is sharp: for instance, suppose that $q$ is not a power of $\rho$, let $\beta\in \ff{q}$ be a $2\rho n$-th primitive root of the unity and $\alpha=\beta^n$ . Consider $M=\begin{pmatrix} 1 &1\\ \alpha& \alpha^{-1}\end{pmatrix}$ and 
$$A=M^{-1} \begin{pmatrix} \beta &0\\ 0&\beta^{-1}\end{pmatrix}M.$$
Observe that $\ord([A])=\rho n$ and if $g$ is the minimum  positive integer such that
$$\beta^{2g}=\frac {uv}{tw}= \frac{\alpha}{\alpha^{-1}}=\beta^{2n},$$
then $g=n=\frac D \rho$, where $t, u, v$ and $w$ are defined as in Lemma \ref{lemma:LI-1}. In the proof of our main result we use the general bound $(g, D)\le \lfloor \frac{D}{2}\rfloor$.
\end{remark}

\section{Bounds for the order of $\langle \theta\rangle\subset {\fff{q}}^*$}

Before the proof of our main result, as in \cite{MartinezReis2016}, we need the following definition:
\begin{define}\label{I_stm}
For each $s,t,m\in\N$,  $m<D$, define the set
\[
I_{s,t,m}:=\left\{(u_0,\ldots,u_{D-1})\in \Z^D\ \left|\ {\sum\limits_{u_j>0}u_j\leq s, \sum\limits_{u_j<0}|u_j|\leq t \quad \text{and}
\atop \text{the first $m$ coordinates are zero}}\right.\right\} 
\]
\end{define}
 \begin{lemma}\label{lemma_I_stm} Let $I_{s,t,m}$ be as in the Definition \ref{I_stm}. Then
$$|I_{s, t, m}|=\sum_{i=0}^{D-m}\binom{D-m}{i}\binom{s}{i}\binom{D-m-i+t}{t}.$$
In particular, for $t\ge \frac{D-m}2$
$$|I_{t, t, m}|> 
\binom{  \frac {D-m}2 +t}{D-m} \binom{2D-2m}{D-m}.$$
\end{lemma}

\begin{proof}
Let us denote $R=D-m$. Notice that, for each $0\le i\le R$ and $0\le j\le R-i$ there are  $\binom{R}{i} \binom{R-i}{j}$ different ways to select $j$ coordinates of $u_m, \dots, u_{D-1}$ to be negative and $i$ coordinates to be positive. In addition, the number of positive  solutions of $x_1+x_2+\cdots+x_i\le s$ is  $\binom{s}{i}$ and the number of positive  solutions of $x_1+x_2+\cdots+x_j\le t$ is  $\binom{t}{j}$. Thus, for each pair  $i, j$, there exist $\binom{R}{i} \binom{R-i}{j} \binom{s}{i}\binom{t}{j}$ elements of $I_{s, t, m}$. Summing over all $i$ and $j$, we obtain 
\begin{equation} \label{eq:C-est}
|I_{s, t, m}|=\sum_{i=0}^{R}\binom{R}{i}\binom{s}{i}\sum_{j=0}^{R-i}\binom{R-i}{j}\binom{t}{j}=\sum_{i=0}^{R}\binom{R}{i}\binom{s}{i}\binom{R-i+t}{t}.
\end{equation}
An easy calculation gives $\binom{s}{i}\binom{R+t-i}{t}=\binom{R}{i}\binom{R-i+t}{R}\frac{\binom{s}{i}}{\binom{t}{i}}$. In particular, if $s=t$ we get

\begin{eqnarray*}|I_{t, t, m}|&=&\sum_{i=0}^R\binom{R}{i}^2\binom{R-i+t}{R}=
\frac 12 \sum_{i=0}^R\binom{R}{i}^2\left[\binom{R-i+t}{R}+\binom{i+t}{R}\right]\\
&\ge&\frac 12\left[\binom{\left \lfloor \frac R2\right\rfloor +t}{R}+\binom{\left\lceil \frac R2\right\rceil +t}{R}\right]  \sum_{i=0}^R\binom{R}{i}^2\\
&= & \frac 12\left[\binom{\left \lfloor \frac R2\right\rfloor +t}{R}+\binom{\left\lceil \frac R2\right\rceil +t}{R}\right] \binom{2R}{R}\\
&\ge& \binom{  \frac {R}2 +t}{R} \binom{2R}{R},
\end{eqnarray*}
where the last inequality follows from the fact that  $\Gamma_N(x):=\binom{x}{N}$ is a convex function for all $x\ge N$.
\end{proof}

\begin{propo}  For every $D\ge 2$ and $r\ge 3$ the following inequalities  hold
\begin{enumerate}[a)]
\item $\displaystyle|I_{\left \lfloor\frac{Dr}2\right\rfloor,\left \lfloor\frac{Dr}2\right\rfloor, 0}|>
\frac1{\sqrt{2}\pi D} \sqrt{\frac {r-1}{r+1}}\cdot\left(\frac { 4(r+1)^{r+1}}{(r-1)^{r-1}}\right)^{\frac D2}  \exp\left(-\frac 1{12D}\cdot\frac {5r^2+3}{r^2-1}\right)$.
\item $\displaystyle|I_{\left \lfloor\frac{Dr}4\right\rfloor,\left \lfloor\frac{Dr}4\right\rfloor, 0}|> 
\frac1{\sqrt{2}\pi D} \sqrt{\frac {r-2}{r+2}}\cdot\left(\frac { (r+2)^{r+2}}{(r-2)^{r-2}}\right)^{\frac D4}  \exp\left(-\frac 5{24D}\cdot\frac {r^2+4}{r^2-4}\right)$.
\item $\displaystyle|I_{\left \lfloor\frac{Dr}4\right\rfloor,\left \lfloor\frac{Dr}4\right\rfloor,\left \lfloor\frac D2\right\rfloor}|>
\frac{\sqrt{2}}{\pi D} \sqrt{\frac {r-1}{r+1}}\cdot\left(\frac { 4(r+1)^{r+1}}{(r-1)^{r-1}}\right)^{\frac D4}  \exp\left(-\frac 1{24D}\cdot\frac {5r^2+3}{r^2-1}\right)$.
\end{enumerate}

\end{propo}

\begin{proof} 
The steps  of the proof are essentially  the same as those used in the proof 
of  \cite[Theorem 2.3]{MartinezReis2016}. 
In fact,
$$\binom{\frac D2+\frac{Dr}4-1}{D}= \frac{\frac D2+\frac {Dr}4-D}{\frac 
D2+\frac{Dr}4} \cdot\binom{D\cdot \frac{r+2}4}{D}=\frac 
{r-2}{r+2}\cdot\binom{D\cdot \frac{r+2}4}{D}.$$
From Corollary 1 in \cite{Sas} 
\begin{align*}
\binom{\frac D2+\frac{Dr}4-1}{D}&\ge\frac {r-2}{r+2}\cdot\sqrt{\frac{\frac {r+2}4}{2\pi \frac {r-2}4}}\left(\frac {\left(\frac {r+2}4\right)^{\frac {r+2}4}}{\left(\frac {r-2}4\right)^{\frac {r-2}4}}\right)^D \frac 1{\sqrt D} \exp\left(-\frac 1{12D}\left(1+\frac{16}{r^2-4}\right)\right)\\
&=\frac1{\sqrt{2\pi D}} \sqrt{\frac {r-2}{r+2}}\cdot\left(\frac { (r+2))^{\frac {r+2}4}}{4(r-2)^{\frac {r-2}4}}\right)^D  \exp\left(-\frac {r^2+12}{12D(r^2-4)}\right).
\end{align*}
Finally,  from  Lemma \ref{lemma_I_stm} and inequality $\binom{2D}{D}> \frac {4^D}{\sqrt{\pi D}}\exp\left(-\frac 1{8D}\right)$, we conclude that
\begin{align*}
|I_{\lfloor \frac{Dr}4\rfloor,\lfloor \frac{Dr}4\rfloor,0}|&\ge \binom{\frac D2+\lfloor\frac{Dr}4\rfloor}{D}\cdot \binom{2D}{D}\ge \binom{\frac D2+\frac{Dr}4-1}{D}\cdot \binom{2D}{D}\\
&>\frac1{\sqrt{2}\pi D} \sqrt{\frac {r-2}{r+2}}\cdot\left(\frac { (r+2)^{r+2}}{(r-2)^{r-2}}\right)^{\frac D4}  \exp\left(-\frac 5{24D}\cdot\frac {r^2+4}{r^2-4}\right).
\end{align*}
By the same process we obtain items a) and c).
\end{proof}


The main result of this paper is consequence of following theorem

\begin{theorem} \label{thm:diagonalizable}
 Let $A \in \GL_2(\ff{q})$, $[A]\ne [I]$  
and $\theta$  be a  generic  root of $F_{A,r}$. 
Then the map
$$
 \begin{array}{rccl}
  \Lambda : & I_{s,t,m} &\longrightarrow&\langle \theta\rangle\\
  &(u_0,\ldots,u_{D-1}) &\longmapsto& \displaystyle\prod_{j=0}^{D-1} \theta^{u_j q^{jr}}
\end{array}
$$
is one to one  in the following cases:
\begin{enumerate}[1) ]
\item $A$ is a triangular matrix , $m=0$  and $s+t<Dr$.
\item $A$ is not a triangular matrix, $(0,1)A^i$ and $(1,0)A^j$ are linearly independent for all $i,j$,   $m=0$  and $s+t<\frac{Dr}2$.
\item  $A$  is not a triangular matrix,  there exists  $0< g<D$ such that $(1,0)$ and $(0,1)A^g$ are linearly dependent, $m=\gcd(g,D)$  and $s+t<\frac{Dr}2$.
\end{enumerate}
\end{theorem}

\begin{proof}
Clearly $I_{s,t,g}\subseteq I_{s,t}$ for any $1\leq g<D$.
For $(u_0,\ldots,u_{D-1})\in I_{s,t}$, we compute
\[
\Lambda(u_0,\ldots,u_{D-1}) 
 = \prod_{j=0}^{D-1} \left(\theta^{q^{jr}}\right)^{u_j} 
  = \prod_{j=0}^{D-1} \left( A^j \circ \theta\right)^{u_j}.
\]
For any matrix $B$ in the class $[A]\in \PGL_2(\fff{q})$, we have
$A^j\circ \theta = B^j\circ \theta$, so we may substittute $A$ with $\delta^{-1} A$,
where $\delta^2 = \det(A)$. This allows us to assume that $\det(A)=1$, with 
$A\in \GL_2(\ff{q^2})$. We have
\[
 \Lambda(u_0,\ldots,u_{D-1}) 
  = \prod_{j=0}^{D-1} \left( A^j \circ \theta\right)^{u_j}
 = \prod_{j=0}^{D-1} \left( \frac{d_j\theta -c_j}{-b_j\theta+a_j}\right)^{u_j}.
\]

Consider now $(u_0,\ldots,u_{D-1}),(v_0,\ldots,v_{D-1})\in I_{s,t}$ and let
$\Lambda(u_0,\ldots,u_{D-1}) = \Lambda(v_0,\ldots,v_{D-1})$. Then we have
\begin{eqnarray*}
&& \prod_{\substack{0\leq j\leq D-1 \\ u_j>0}} \left( d_j \theta -c_j\right)^{u_j}
\prod_{\substack{0\leq j\leq D-1 \\ u_j<0}} \left(-b_j \theta +a_j\right)^{-u_j}
\prod_{\substack{0\leq j\leq D-1 \\ v_j<0}} \left( d_j \theta -c_j\right)^{-v_j}
\prod_{\substack{0\leq j\leq D-1 \\ v_j>0}} \left(-b_j \theta +a_j\right)^{v_j} \\
&=& \prod_{\substack{0\leq j\leq D-1 \\ v_j>0}} \left( d_j \theta -c_j\right)^{v_j}
\prod_{\substack{0\leq j\leq D-1 \\ v_j<0}} \left(-b_j \theta +a_j\right)^{-v_j}
\prod_{\substack{0\leq j\leq D-1 \\ u_j<0}} \left( d_j \theta -c_j\right)^{-u_j}
\prod_{\substack{0\leq j\leq D-1 \\ u_j>0}} \left(-b_j \theta +a_j\right)^{u_j}.
\end{eqnarray*}
So, $\theta$ is a root of $F(X)-G(X)$, where 
\begin{eqnarray*}
  F(X) &=& \prod_{\substack{0\leq j\leq D-1 \\ u_j>0}} \left( d_j X -c_j\right)^{u_j}
\prod_{\substack{0\leq j\leq D-1 \\ u_j<0}} \left(-b_j X +a_j\right)^{-u_j}
\prod_{\substack{0\leq j\leq D-1 \\ v_j<0}} \left( d_j X -c_j\right)^{-v_j}
\prod_{\substack{0\leq j\leq D-1 \\ v_j>0}} \left(-b_j X +a_j\right)^{v_j} \\
G(X) &=& \prod_{\substack{0\leq j\leq D-1 \\ v_j>0}} \left( d_j X -c_j\right)^{v_j}
\prod_{\substack{0\leq j\leq D-1 \\ v_j<0}} \left(-b_j X +a_j\right)^{-v_j}
\prod_{\substack{0\leq j\leq D-1 \\ u_j<0}} \left( d_j X -c_j\right)^{-u_j}
\prod_{\substack{0\leq j\leq D-1 \\ u_j>0}} \left(-b_j X +a_j\right)^{u_j}.
\end{eqnarray*}
We consider the following three cases:

{\bf Case 1:} Suppose that  $A$ is a triangular matrix.
Observe that if $\theta$ is root of $F_{A,r}(x)$, then $\theta^{-1}$ is root of the polynomial $F_{B,r}(x)$ where $B=\begin{pmatrix} d&c\\ b&a \end{pmatrix}$. Therefore, changing $\theta$ by $\theta^{-1}$, we can suppose, without loss of generality that $A$ is lower triangular matrix.  Thus $b_j=0$ for all $j$ and the degree of the polynomial $F(X)$ and $G(X)$ are respectively
$$\sum_{u_j\ge 0}u_j -\sum_{v_j\ge 0}v_j\le s+t \quad \text{and}\quad \sum_{v_j\ge 0}v_j -\sum_{u_j\ge 0}u_j\le s+t.
$$
Since $\deg(F(X))\le s+t<Dr$ and $\deg(G(X))\le s+t<Dr$ and $F(X)-G(X)$ is divisible by the minimal irreducible polynomial that $\theta$ is root, that has degree $Dr$, it follows that $F(X)=G(X)$.  In particular, these polynomials have the same root of the same order, then, from Lemma~\ref{lemma:LI-3} we conclude that 
$(u_0,\ldots,u_{D-1})=(v_0,\ldots,v_{D-1})$, that is, $\Lambda$ is injective.

{\bf Case 2:} For any $0\leq k,n<D$, from Lemma \ref{lemma:LI-1} the vectors $(c_n,d_n), (a_k,b_k)$ are linearly independent.
From the condition that $s+t\le \frac D2$, it follows that $\deg(F),\deg(G)<Dr$, and therefore,
we have $F(X)=G(X)$. Then, the result follows the same way that the case 1.  

{\bf Case 3:} There exist $0\leq k,n<D$, such that $(c_n,d_n)=\gamma(a_k,b_k)$, for some $\gamma \in \ff{q^2}^{*}$. Let us define $g=n-k$ and $m=\gcd(g,D)-1$. 
In this case, it turns out that we have to restrict $\Lambda$ to the set $I_{s,t,m}$ to maintain
injective. Indeed, by Lemma~\ref{lemma:LI-2}, we have
\[
 d_jX-c_j = \epsilon_j \gamma (b_{j-g}X-a_{j-g}), \ \ \mbox{ for }\ 0\leq j\leq D-1
\]
and we obtain
\begin{eqnarray*}
F(X) &=& \epsilon_F \gamma^{e_F} \prod_{u_j<0} (b_jX-a_j)^{-u_j} \prod_{v_j>0} (b_jX-a_j)^{v_j}
                             \prod_{u_j>0} (b_{j-g}X-a_{j-g})^{u_j} \prod_{v_j<0} (b_{j-g}X-a_{j-g})^{-v_j}\\
G(X) &=& \epsilon_G \gamma^{e_G} \prod_{v_j<0} (b_jX-a_j)^{-v_j} \prod_{u_j>0} (b_jX-a_j)^{u_j}
                             \prod_{v_j>0} (b_{j-g}X-a_{j-g})^{v_j} \prod_{u_j<0} (b_{j-g}X-a_{j-g})^{-u_j},
\end{eqnarray*}
where $\epsilon_F,\epsilon_G\in\{-1,1\}$, $e_F=\sum_{u_j>0} u_j-\sum_{v_j<0}v_j$ and 
$e_G=\sum_{v_j>0} v_j-\sum_{u_j<0}u_j$.
By the definition of $I_{s,t,m}$, again we have $\deg(F),\deg(G)<Dr$, so that $F(X)=G(X)$, and we obtain
\[
\epsilon \gamma^{e_G-e_F}\prod_{j=0}^{D-1} (b_jX-a_j)^{u_j-u_{j+g}} = \prod_{j=0}^{D-1}(b_jX-a_j)^{v_j-v_{j+g}},
\]
with $\epsilon\in\{-1,1\}$. By Lemma~\ref{lemma:LI-1}, we obtain
\[
u_j-u_{j+g} = v_j-v_{j+g}, \ \ 0\leq j\leq D-1.
\]
Let us define $x_j = u_j-v_j$, $0\leq j<D$. Then we have $x_{j+g}=x_j$ for $j\geq 0$ (where we take the indices mod $D$).
Let $J=\{\overline{j}\ :\ x_j=0\}$. We know that $\{\overline{0},\ldots,\overline{(g,D)-1}\}\subseteq J$ 
and the recursion gives us that
$\left\{\overline{a+ig}\ :\ 0\leq a < (g,D),\ i\geq 0  \right\} \subseteq J$. It is easy to see that
$J=\Z_D$, therefore $(u_0,\ldots,u_{D-1})=(v_0,\ldots,v_{D-1})$ and $\Lambda$ is injective.
\end{proof}

\end{document}